\documentclass[reqno,11pt]{amsart}
\usepackage{amssymb}
\usepackage{mathrsfs}
\usepackage{amsmath,amssymb}
\usepackage{paralist}
\usepackage{graphics} 
\usepackage{epsfig} 
\usepackage[colorlinks=true]{hyperref}
\hypersetup{urlcolor=red, citecolor=blue}
\usepackage[top=1in,bottom=1in,left=1in,right=1in]{geometry}
\usepackage{cite}
\newtheorem{thm}{Theorem}[section]

\newtheorem{lem}[thm]{Lemma}

\theoremstyle{definition}
\newtheorem{defn}{Definition}[section]
\newtheorem{rem}{Remark}[section]
\numberwithin{equation}{section}
\newcommand{\be}{\begin{equation}}
\newcommand{\ee}{\end{equation}}
\begin{document}
\title[nonlinear Keller-Segel equation with singular sensitivity]
{Global generalized solutions to a nonlinear Keller-Segel equation with singular sensitivity}%
\author[Yan]{Jianlu Yan}%
\address{Department of Mathematics, Southeast University, Nanjing 210096, P. R. China}
\email{230159430@seu.edu.cn}
\author[Li]{Yuxiang Li}%
\address{Department of Mathematics, Southeast University, Nanjing 210096, P. R. China}
\email{lieyx@seu.edu.cn}
\thanks{Supported in part by National Natural Science Foundation of China (No. 11171063, No. 11671079, No. 11701290), National Natural Science Foundation of China under Grant (No. 11601127), and National Natural Science Foundation of Jiangsu Provience (No. BK20170896).}


\begin{abstract}
We consider the chemotaxis system
\begin{eqnarray*}
\begin{cases}
\begin{array}{lll}
\medskip
u_t =\Delta u^m - \nabla(\frac{u}{v}\nabla v),&{} x\in\Omega,\ t>0,\\
\medskip
v_t =\Delta v -uv,&{}x\in\Omega,\ t>0,\\
\medskip
\frac{\partial u}{\partial \nu}=\frac{\partial v}{\partial\nu}=0,&{}x\in\partial\Omega,\ t>0,\\
\medskip
u(x,0)=u_0(x),\ v(x,0)=v_0(x), &{}x\in\Omega,
\end{array}
\end{cases}
\end{eqnarray*}
in a smooth bounded domain $\Omega\subset \mathbb{R}^n$, $n\geq2$. In this work it is shown that for all reasonably regular initial data $u_0\geq0$ and $v_0>0$, the corresponding Neumann initial-boundary value problem possesses a global generalized solution provided that $m>1+\frac{n-2}{2n}$.
\end{abstract}
\maketitle
\section{Introduction}
In this paper, we consider the following chemotaxis system with singular sensitivity:
\begin{eqnarray}\label{e11}
\begin{cases}
\begin{array}{lll}
\medskip
u_t =\Delta u^m - \nabla(\frac{u}{v}\nabla v),&{} x\in\Omega,\ t>0,\\
\medskip
v_t =\Delta v -uv,&{}x\in\Omega,\ t>0,\\
\medskip
\frac{\partial u}{\partial \nu}=\frac{\partial v}{\partial\nu}=0,&{}x\in\partial\Omega,\ t>0,\\
\medskip
u(x,0)=u_0(x),\ v(x,0)=v_0(x), &{}x\in\Omega,
\end{array}
\end{cases}
\end{eqnarray}
where $\Omega\subset \mathbb{R}^n$, $n\geq2$ is a bounded domain with smooth boundary; $u=u(x,t)$ denotes the density of the cells; $v=v(x,t)$ represents the concentration of the oxygen; the initial data $u_0$ and $v_0$ are assumed to be nonnegative functions.

Chemotaxis is the influence of chemical substances in the environment on the movement of mobile species. This can lead to strictly oriented movement or to partially oriented and partially tumbling movement. The movement towards a higher concentration of the chemical substance is termed positive chemotaxis and the movement towards regions of lower chemical concentration is called negative chemotactical movement. Chemotaxis is an important means for cellular communication. Communication by chemical signals determines how cells arrange and organize themselves, like for instance in development or in living tissues\cite{Horstmann&Winkler-JDE-2005}.

System (\ref{e11}) is a well-known chemotaxis model and was initially introduced by Keller and Segel \cite{Keller1970Initiation} to describe the aggregation of cellular slime molds Dictyostelium discoideum. When $m=1$, (\ref{e11}) is a standard Keller-Segel system
\begin{eqnarray}\label{e1.2}
\begin{cases}
\begin{array}{lll}
\medskip
u_t =\Delta u - \chi\nabla(\frac{u}{v}\nabla v),&{} x\in\Omega,\ t>0,\\
\medskip
v_t =\Delta v -uv,&{}x\in\Omega,\ t>0,\\
\end{array}
\end{cases}
\end{eqnarray}
where the second equation models consumption of the signal upon contact with cells, and where in the first equation it is assumed that the chemotactic stimulus is perceived in accordance with the Weber-Fechner law, thus requiring the chemotactic sensitivity $S(u,v):=\chi\frac{u}{v}$ to be chosen proportional to the reciprocal
signal density. Indeed, the ability of this particular type of absorption-taxis interplay to generate wave-like solution behavior formed a motivation for several analytical studies on the existence and stability properties of traveling wave solutions to (\ref{e1.2})(see \cite{Jin2013Asymptotic,Tong2009NONLINEAR,TONG2010NONLINEAR}) and also to some closely related systems(see \cite{Meyries2008Local,Nagai1991Traveling,Schwetlick2010Traveling}). The singular chemotactic sensitivities as in (\ref{e1.2}) are very important in biology, which have been underlined independently in modeling approaches (see \cite{Levine1997A,Othmer1997Aggregation,Kalinin2009Logarithmic}) and in tumor angiogenesis (see \cite{Sleeman2001Partial,Levine2000A}) and also in taxis-driven morphogen transport (see \cite{Bollenbach2007Morphogen}).

Recently, Winkler \cite{Winkler2016The} showed that in bounded planar domains $\Omega$ with smooth boundary, for all reasonably regular initial data $u_0>0$ and $v_0>0$, the corresponding Neumann initial-boundary value problem possesses a global generalized solution. In higher-dimensional domains, in \cite{Winkler2018Renomalized} the same author constructed  renormalized solutions in a radially symmetric setting. When the term $\Delta u^m$ is replaced by $\nabla\cdot(D(u)\nabla u)$ in the first equation in (\ref{e11}), Lankeit\cite{Lankeit2016Locally} showed the existence of locally bounded global solutions of the system with $D(u)\geq\delta u^{m-1}$, $\delta>0$, $n\geq 2$, under the condition $m>1+\frac{n}{4}$.

In contrast to this, in the related chemotaxis system
\begin{eqnarray}\label{e}
\begin{cases}
\begin{array}{lll}
\medskip
u_t =\Delta u-\chi\nabla(\frac{u}{v}\nabla v),&{} x\in\Omega,\ t>0,\\
\medskip
v_t =\Delta v-v+u,&{}x\in\Omega,\ t>0,\\
\end{array}
\end{cases}
\end{eqnarray}
$v$ does not stand for a nutrient to be consumed but a signalling substance produced by the bacteria themselves, i.e. the evolution is governed by the singularity in the sensitivity function is mitigated by v tending to stay away from 0 thanks to the production term in the second equation. There have been many authors studying system (\ref{e}). Winkler \cite{Winkler2011Global} proved that if $0<\chi<\sqrt{\frac{2}{n}}$ then there exists a global-in-time classical solution, generalizing a previous result which asserts the same for n = 2 only. Furthermore, it is seen that the range of admissible $\chi$ can be enlarged upon relaxing the solution concept. More precisely, global existence of weak solutions is established whenever $0<\chi<\sqrt{\frac{n+2}{3n-4}}$.
In \cite{Stinner2011Global} Stinner and Winkler proved that in the two-dimensional radially symmetric case certain generalized solutions can be constructed actually without any restriction on the size of $\chi$. Recently, it is proved in Lankeit and Winkler \cite{Lankeit2017A} that the problem possesses a global
generalized solution provided
\begin{equation*}
\chi=\begin{cases}
\infty,&\text{if $n=2$};\\
\sqrt{8},&\text{if $n=3$};\\
\frac{n}{n-2},&\text{if $n\geq 4$}.
\end{cases}
\end{equation*}
When the first equation of (\ref{e}) is replaced by $u_t=\Delta u-\nabla\cdot\left(u\chi(v)\nabla v\right)$, in \cite{Winkler2010Absence}, the chemotactic sensitivity function is assumed to generalize the prototype $\chi(v)=\frac{\chi_0}{(1+\alpha v)^2}$, $v>0$. It is proved that no chemotactic collapse occurs in the sense that for any choice of nonnegative initial data $u(\cdot,0)\in C^0(\bar{\Omega})$ and $v(\cdot,0)\in W^{1,r}$ (with some $r>n$), the corresponding initial-boundary value problem possesses a unique global solution that is uniformly bounded. \cite{FUJIE2014140} presents global existence and boundedness of classical solutions to the fully parabolic chemotaxis system with the strongly singular sensitivity function $\chi(v)$ such that $0<\chi(v)<\frac{\chi_0}{v^k}(\chi_0>0,k>1)$.

Motivated by the above works, the goal of this paper is to deal with global  generalized solvability in the high-dimensional version of system (\ref{e11}) for arbitrary large initial data in an appropriate framework. We now state the main results of this paper.

\textbf{Main results.}
To be precise, we will only assume the initial data $u_0$ and $v_0$ to satisfy
\begin{equation}\label{e12}
\begin{cases}
\begin{aligned}
&u_0\in C^0(\bar{\Omega})\ \text{with}\ u_0\geqslant 0\ \text{in}\ \Omega \ \text{and}\ u_0\not\equiv 0\ \text{as well as}\\
&v_0\in W^{1,\infty}(\Omega)\ \text{with}\ v_0>0\ \text{in}\ \bar{\Omega}
\end{aligned}
\end{cases}
\end{equation}
Under this assumption, we shall firstly see that actually no further requirements on the size of the data are necessary for the construction of certain globally defined generalized solutions.

\begin{thm}\label{main result}
Let $\Omega\subset\mathbb{R}^n$, $n\geq 2$ be a bounded domain with smooth boundary. Then for all $u_0$ and $v_0$ satisfying (\ref{e12}), the problem (\ref{e11}) for $m>1+\frac{n-2}{2n}$ possesses at least one global generalized solution in the sense of Definition \ref{defi21}.
\end{thm}

The rest of this paper is organized as follows. In Section 2, we introduce the conception of generalized solution. In Section 3, we give some priori estimates and state some straightforward consequences. Section 4 is devoted to showing the strong convergence of the solution of regularized problems. Finally, we give the proof of the main result in Section 5.

\vskip 3mm
\section{Contraction of Global Generalized Solutions}
\begin{defn}\label{defi21}
Assume that $u_0$ and $v_0$ satisfy $(\ref{e12})$. Then a pair $(u,v)$ of functions
\begin{equation}\label{e21}
\begin{cases}
\begin{aligned}
&u\in L^1_{\rm{loc}}(\bar{\Omega}\times[0,\infty)),\\
&v\in L^{\infty}_{\rm{loc}}(\bar{\Omega}\times[0,\infty))\cap L^2_{\rm{loc}}([0,\infty);W^{1,2}(\Omega)),
\end{aligned}
\end{cases}
\end{equation}
with
\begin{equation}\label{e22}
 u\geq 0\ \text{a.e. in}\ \Omega \times (0,\infty)\ \text{and}\ v>0\ \text{a.e. in}\ \Omega \times (0,\infty)
\end{equation}
as well as
\begin{equation}\label{e23}
\nabla u^{m-1}\in L^2_{\rm{loc}}(\bar{\Omega}\times[0,\infty))\ \text{and}\ \nabla \ln v\in\ L^2_{\rm{loc}}(\bar{\Omega}\times[0,\infty)),
\end{equation}
will be called a \emph{global generalized solution }of (1.4) if $u$ has the mass conservation property
\begin{equation}\label{e24}
\int_{\Omega}u(x,t)dx=\int_{\Omega}u_0(x)\ \ \text{for a.e.}\ t>0,
\end{equation}
if the inequality
\begin{equation}\label{e25}
\begin{split}
&-\frac{m-1}{m(m-2)}\int_{0}^{\infty}\int_{\Omega}u^{m-1}\varphi_t-\frac{m-1}{m(m-2)}\int_{\Omega}u_0^{m-1}\varphi(\cdot,0) \\
&\leq-\int_{0}^{\infty}\int_{\Omega}|\nabla u^{m-1}|^2\varphi- \frac{m-1}{m-2}\int_{0}^{\infty}\int_{\Omega}u^{m-1}\nabla u^{m-1}\cdot\nabla\varphi\\
&+\frac{m-1}{m}\int_{0}^{\infty}\int_{\Omega}(\nabla u^{m-1}\cdot\nabla \ln v)\varphi+\frac{(m-1)^2}{m(m-2)}\int_{0}^{\infty}\int_{\Omega}u^{m-1}\nabla \ln v\cdot\nabla\varphi
\end{split}
\end{equation}
holds for each non-negative $\varphi\in C_{0}^{\infty} (\bar{\Omega}\times [0,\infty))$, and if moreover the identity
\begin{equation}\label{e26}
\int_{0}^{\infty}\int_{\Omega}v\varphi_t+\int_{\Omega}v_0\varphi(\cdot,0)=\int_{0}^{\infty}\int_{\Omega}\nabla v\cdot\nabla\varphi+\int_{0}^{\infty}\int_{\Omega}uv\varphi
\end{equation}
is valid for any $\varphi\in L^{\infty}(\Omega\times(0,\infty))\cap L^2((0,\infty);W^{1,2}(\Omega))$ having compact support in $\bar{\Omega}\times[0,\infty)$ with $\varphi_t\in L^2(\Omega\times(0,\infty))$.
\end{defn}

\begin{rem}\label{remone}
When $m=2$, (\ref{e25}) should be replaced by the identity
\begin{equation}\label{eremone}
-\frac{1}{2}\int_{0}^{\infty}\int_{\Omega}u\varphi_t-\frac{1}{2}\int_{\Omega}u_0\varphi(\cdot,0)
=-\int_{0}^{\infty}\int_{\Omega}u\nabla u\cdot\nabla\varphi
+\frac{1}{2}\int_{0}^{\infty}\int_{\Omega}u\nabla \ln v\cdot\nabla\varphi.
\end{equation}
\end{rem}

In order to construct such generalized solutions by means of an approximation procedure, throughout the sequel we fix a nonincreasing cut-off function $\rho\in C^{\infty}([0,\infty))$ fulfilling $\rho\equiv 1$ in $[0,1]$ and $\rho\equiv 0$ in $[2,\infty)$, and define $f_{\varepsilon}\in C^{\infty}([0,\infty))$ by letting
\begin{equation}\label{e27}
f_{\varepsilon}(s):=\int_{0}^{s}\rho(\varepsilon\sigma)d\sigma,\ \ s\geq 0,
\end{equation}
for $\varepsilon\in(0,1)$. Then for any such $\varepsilon$, the properties of $\rho$ imply that $f_{\varepsilon}$ evidently satisfies
\begin{equation}\label{e28}
f_{\varepsilon}(0)=0\ \ \text{and}\ 0\leq f'_{\varepsilon}\leq 1\ \ \text{on}\ [0,\infty)
\end{equation}
as well as
\begin{equation}\label{e29}
f_{\varepsilon}(s)=s\ \ \text{for all}\ s\in[0,\frac{1}{\varepsilon}]\ \ \text{and}\ f'_{\varepsilon}(s)=0\ \ \text{for all}\ s\geq\frac{2}{\varepsilon},
\end{equation}
and moreover we have
\begin{equation*}
f_{\varepsilon}(s)\nearrow s\ \ \text{and}\ \ f'_{\varepsilon}(s)\nearrow 1\ \text{as}\ \ \varepsilon\searrow 0
\end{equation*}
for each $s\geq 0$.

Now as a consequence of $(\ref{e29})$, each of the approximate problems:

\begin{equation}\label{e210}
\begin{cases}
\begin{array}{lll}
\medskip
u_{\varepsilon t}=\nabla(D_\varepsilon(u_{\varepsilon})\nabla u_{\varepsilon})-\nabla\cdot(\frac{u_{\varepsilon} f'_{\varepsilon}(u_{\varepsilon})}{v_{\varepsilon}}\nabla v_{\varepsilon}), \ \ \ &x\in \Omega,\ t>0,\\
\medskip
v_{\varepsilon t}=\Delta v_{\varepsilon}-f_{\varepsilon}(u_{\varepsilon})v_{\varepsilon}, \ &x\in \Omega,\ t>0,\\
\medskip
\frac{\partial u_{\varepsilon}}{\partial \nu}=\frac{\partial v_{\varepsilon}}{\partial \nu}=0,\ &x\in \partial \Omega,\ t>0,\\
\medskip
u_{\varepsilon}(x,0)=u_0(x),\ \ v_{\varepsilon}(x,0)=v_0(x) \ \ \ &x\in \Omega\\
\end{array}
\end{cases}
\end{equation}
is in fact globally solvable, where $D_\varepsilon(s)=D(s+\varepsilon)$, $D(s)=ms^{m-1}$.

~\\
\begin{lem}\label{lem2.1}
Assume that $(\ref{e12})$ holds,
and let $\varepsilon\in (0,1)$. Then $(\ref{e210})$ possesses a global classical solution $(u_{\varepsilon},v_{\varepsilon})$, for each $\vartheta>2$ uniquely determined by the inclusions
\begin{equation*}
\begin{cases}
\begin{aligned}
&u_{\varepsilon}\in C^0(\bar{\Omega}\times [0,\infty))\cap C^{2,1}(\bar{\Omega}\times (0,\infty)),\\
&v_{\varepsilon}\in C^0(\bar{\Omega}\times [0,\infty))\cap C^{2,1}(\bar{\Omega}\times (0,\infty))\cap L_{\rm{loc}}^{\infty}([0,\infty);W^{1,\vartheta}(\Omega)),
\end{aligned}
\end{cases}
\end{equation*}
which is such that $u_{\varepsilon}>0$ in $\bar{\Omega}\times (0,\infty)$ and
\begin{equation}\label{e211}
\int_{\Omega}u_{\varepsilon}(x,t)dx=\int_{\Omega}u_0(x)dx\ \ \text{for all}\ t>0
\end{equation}
as well as
\begin{equation}\label{e212}
0<v_{\varepsilon}\leq\|v_0\|_{L^{\infty}(\Omega)}\ \ \ \text{in}\ \bar{\Omega}\times [0,\infty).
\end{equation}
\end{lem}
\begin{proof}
By straightforward adaptation of well-known arguments, as detailed e.g. in \cite{Horstmann&Winkler-JDE-2005} and \cite{Yokota2015Stabilization} for closed related situations, it can be verified that for each $\varepsilon\in(0,1)$ there exist $T_{max,\varepsilon}\in(0,\infty]$ and a unique couple of functions
\begin{equation*}
\begin{cases}
\begin{aligned}
&u_{\varepsilon}\in C^0(\bar{\Omega}\times [0,T_{max,\varepsilon}))\cap C^{2,1}(\bar{\Omega}\times (0,T_{max,\varepsilon})),\\
&v_{\varepsilon}\in C^0(\bar{\Omega}\times [0,T_{max,\varepsilon}))\cap C^{2,1}(\bar{\Omega}\times (0,T_{max,\varepsilon}))\cap L_{\rm{loc}}^{\infty}([0,T_{max,\varepsilon});W^{1,\vartheta}(\Omega)),
\end{aligned}
\end{cases}
\end{equation*}
with $u_{\varepsilon}>0$ in $\bar{\Omega}\times(0,T_{max,\varepsilon})$ and $v_{\varepsilon}>0$ in $\bar{\Omega}\times[0,T_{max,\varepsilon})$, and such that
\begin{equation}\label{e213}
\begin{split}
\text{either}\ \ T_{max.\varepsilon}=\infty,\ &\text{or}\ \ \limsup\limits_{t\nearrow T_{max,\varepsilon}}(\|u_{\varepsilon}(\cdot,t)\|_{L^{\infty}(\Omega)}+\|v_{\varepsilon}(\cdot,t)\|_{W^{1,\vartheta}(\Omega)})=\infty,\\
&\text{or}\ \ \limsup\limits_{t\nearrow T_{max,\varepsilon}}\inf_{x\in\Omega}v_{\varepsilon}(x,t)=0.
\end{split}
\end{equation}
Moreover, an intergation of the first equation in $(\ref{e210})$ over $x\in\Omega$ shows that this solution enjoys the mass conservation property
\begin{equation*}
\frac{d}{dt}\int_{\Omega}u_{\varepsilon}=0\ \ \text{for all}\ t\in(0,T_{\max,\varepsilon}),
\end{equation*}
whereas from the non-negativity of $f_\varepsilon$ and the maximum principle it follows that
\begin{equation*}
\|v_\varepsilon(\cdot,t)\|_{L^\infty(\Omega)}\leq\|v_0\|_{L^\infty(\Omega)}\ \ \text{in}\ \bar{\Omega}\times(0.T_{\max,\varepsilon}).
\end{equation*}

To prove the lemma, we thus only need to verify that for any fixed $\varepsilon\in(0,1)$, the corresponding maximal existence time $T_{\max,\varepsilon}$ can not be finite, which amounts to showing that in $(\ref{e213})$, neither the second nor the third alternative can occur under the contrary hypothesis that $T_{\max,\varepsilon}<\infty$. But since supp$f'_\varepsilon\subset[0,\frac{2}{\varepsilon}]$ by $(\ref{e29})$,an application of the maximum principle to the first equation in $(\ref{e210})$ shows that
\begin{equation}\label{e214}
u_\varepsilon(x,t)\leq c_1(\varepsilon):=\max\big\{\|u_0\|_{L^\infty(\Omega)},\frac{2}{\varepsilon}\big\}\ \ \text{for all}\ x\in\Omega\ \text{and}\ t\in(0,T_{\max,\varepsilon}).
\end{equation}
Together with $(\ref{e212})$ and argument from parabolic regularity theory, this firstly ensures that for each $\tau\in(0,T_{\max,\varepsilon})$ the number $\sup_{t\in(\tau,T_{\max,\varepsilon})}\|v_\varepsilon(\cdot,t)\|_{W^{1,\vartheta}}$ is finite. Secondly, by comparison in the second equation in $(\ref{e210})$ we moreover obtain from $(\ref{e214})$ that
\begin{equation*}
v_{\varepsilon}(x,t)\geq\big\{\min\limits_{y\in\Omega}v_0(y)\big\}\cdot e^{-c_1(\varepsilon)t}\ \ \text{for all}\ x\in\Omega\ \text{and}\ t\in(0,T_{\max,\varepsilon}),
\end{equation*}
which also excludes the rightmost alternative in $(\ref{e213})$ and thereby completes the proof.
\end{proof}

Now following a standard procedure of changing variables in $(\ref{e11})$, taking $u_{\varepsilon}$ and $v_{\varepsilon}$ from Lemma \ref{lem2.1} we substitute
\begin{equation}\label{e215}
w_{\varepsilon}:=-\ln (\frac{v_{\varepsilon}(x,t)}{\|v_0\|_{L^{\infty}(\Omega)}}),\ \ (x,t)\in \bar{\Omega}\times [0,\infty),\ \ \varepsilon\in (0,1),
\end{equation}
and thus infer from the latter that each of the problems
\begin{equation}\label{e216}
\begin{cases}
\begin{array}{lll}
\medskip
u_{\varepsilon t}=\nabla(D_\varepsilon(u_{\varepsilon})\nabla u_{\varepsilon})+\nabla\cdot(u_{\varepsilon} f'_{\varepsilon}(u_{\varepsilon})\nabla w_{\varepsilon}), \ \ \ &x\in \Omega,\ t>0,\\
\medskip
w_{\varepsilon t}=\Delta w_{\varepsilon}-|\nabla w_{\varepsilon}|^2+f_{\varepsilon}(u_{\varepsilon}), \ &x\in \Omega,\ t>0,\\
\medskip
\frac{\partial u_{\varepsilon}}{\partial \nu}=\frac{\partial w_{\varepsilon}}{\partial \nu}=0,\ &x\in \partial \Omega,\ t>0,\\
\medskip
u_{\varepsilon}(x,0)=u_0(x),\\
\medskip
w_{\varepsilon}(x,0)=w_0(x):=-\ln (\frac{v_{0}(x,t)}{\|v_0\|_{L^{\infty}(\Omega)}}) \ \ \ &x\in \Omega,
\end{array}
\end{cases}
\end{equation}
possesses a global classical solution $(u_{\varepsilon},w_{\varepsilon})$ for which both $u_{\varepsilon}$ and $w_{\varepsilon}$ are nonnegative.

\section{A priori estimates}
\begin{lem}\label{lem2.2}
For all $\varepsilon\in (0,1)$, we have
\begin{equation}\label{e217}
\int_{\Omega}w_{\varepsilon}(\cdot,t)+\int_{0}^{t}\int_{\Omega}|\nabla w_{\varepsilon}|^2\leq\int_{\Omega}w_{0}+\lambda t\ \ \text{for all}\ t>0,
\end{equation}
where $\lambda:=\int_{\Omega}u_0$. In particular
\begin{equation}\label{e218}
\|w_{\varepsilon}(\cdot,t)\|_{L^1(\Omega)}\leq\int_{\Omega}w_0+\lambda t\ \ \text{for all}\ t>0
\end{equation}
and
\begin{equation}\label{e219}
\int_{0}^{t}\int_{\Omega}|\nabla w_{\varepsilon}|^2\leq\int_{\Omega}w_0+\lambda t\ \ \text{for all}\ t>0,
\end{equation}
where $\lambda:=\int_{\Omega}u_0$.
\end{lem}

\begin{proof}
An integration of second equation in $(\ref{e216})$ over $\Omega$ shows that
\begin{equation*}
\frac{d}{dt}\int_{\Omega}w_{\varepsilon}=-\int_{\Omega}|\nabla w_{\varepsilon}|^2+\int_{\Omega}f_{\varepsilon}(u_{\varepsilon})\ \ \text{for all}\ t>0.
\end{equation*}
Since $\int_{\Omega}f_{\varepsilon}(u_{\varepsilon}(\cdot,t))\leq\int_{\Omega}u_{\varepsilon}(\cdot,t)=\lambda$ for all $t>0$ by $(\ref{e27})$ and $(\ref{e211})$, this immediately yields $(\ref{e217})$. Whereupon thanks to the non-negativity of $w_{\varepsilon}$, both $(\ref{e218})$ and $(\ref{e219})$ are evident consequences thereof.
\end{proof}
\begin{lem}\label{lem2.3}
Let $m>1$. Then for all $\varepsilon\in (0,1)$, there exists $C(m)>0$ such that
\begin{equation}\label{e220}
\int_{0}^{t}\int_{\Omega}|\nabla (u_{\varepsilon}+\varepsilon)^{m-1}|^2\leq C(m)\left(\int_{\Omega}w_0+\lambda t+\lambda\right)\ \ \text{for all}\ t>0,
\end{equation}
where $\lambda :=\int_{\Omega}u_0$.
\end{lem}
\begin{proof}
We distinguish between two possible cases.

Case 1: $m=2$. We multiply the first equation in $(\ref{e216})$ by $\ln(u_{\varepsilon}+\varepsilon)$, and integrate by parts to find that
\begin{equation*}
\begin{split}
&\int_\Omega(u_\varepsilon(\cdot,t)+\varepsilon)(\ln(u_\varepsilon(\cdot,t)+\varepsilon)-1)-\int_\Omega(u_0+\varepsilon)(\ln(u_0+\varepsilon)-1)\\
&=-2\int_0^\infty\int_\Omega|\nabla u_\varepsilon|^2-\int_0^\infty\int_\Omega\frac{u_\varepsilon f'_\varepsilon(u_\varepsilon)}{u_\varepsilon+\varepsilon}\nabla u_\varepsilon\cdot\nabla w \ \ \text{for all}\ t>0.
\end{split}
\end{equation*}
Therefore, using Cauchy-Schwarz inequality we can infer that
\begin{equation*}
\begin{split}
2\int_0^\infty\int_\Omega|\nabla u_\varepsilon|^2&\leq\int_\Omega(u_0+\varepsilon)(\ln(u_0+\varepsilon)-1)-2\int_0^\infty\int_\Omega|\nabla u_\varepsilon|^2-\int_0^\infty\int_\Omega\frac{u_\varepsilon f'_\varepsilon(u_\varepsilon)}{u_\varepsilon+\varepsilon}\nabla u_\varepsilon\cdot\nabla w\\
&\leq\int_\Omega(u_0+\varepsilon)(\ln(u_0+\varepsilon)-1)+\int_0^\infty\int_\Omega|\nabla u_\varepsilon||\nabla w_\varepsilon|\\
&\leq\int_\Omega(u_0+\varepsilon)(\ln(u_0+\varepsilon)-1)+\frac{1}{2}\int_0^\infty\int_\Omega|\nabla u_\varepsilon|^2+\frac{1}{2}\int_0^\infty\int_\Omega|\nabla w_\varepsilon|^2\ \ \text{for all}\ t>0,
\end{split}
\end{equation*}
which implies $(\ref{e220})$ due to $(\ref{e12})$ and Lemma \ref{lem2.2}.

Case 2: $m\neq 2$. We multiply the first equation in $(\ref{e216})$ by $(u_{\varepsilon}+\varepsilon)^{m-2}$, and integrate by parts to find that
\begin{equation}\label{e221}
\begin{split}
&\int_{\Omega}(u_{\varepsilon}(\cdot,t)+\varepsilon)^{m-1}-\int_{\Omega}(u_{0\varepsilon}+\varepsilon)^{m-1}\\
&=-\frac{m(m-2)}{m-1}\int_{0}^{t}\int_{\Omega}|\nabla (u_{\varepsilon}+1)^{m-1}|^2\\
&\quad-(m-1)(m-2)\int_{0}^{t}\int_{\Omega}f'_{\varepsilon}(u_{\varepsilon})u_{\varepsilon}(u_{\varepsilon}+\varepsilon)^{m-3}\nabla w_{\varepsilon}\cdot\nabla u_{\varepsilon}\ \ \text{for all}\ t>0.
\end{split}
\end{equation}

Next, we divide this case into two subcases:

Subcase 1: $m>2$. Using Cauchy-Schwarz inequality because of $(\ref{e12})$ we see that
\begin{equation*}
\begin{split}
&\frac{m(m-2)}{m-1}\int_{0}^{t}\int_{\Omega}|\nabla (u_{\varepsilon}+\varepsilon)^{m-1}|^2\\
&\leq\int_{\Omega}(u_{0\varepsilon }+\varepsilon)^{m-1}-(m-1)(m-2)\int_{0}^{t}\int_{\Omega}f'_{\varepsilon}(u_{\varepsilon})u_{\varepsilon}(u_{\varepsilon}+\varepsilon)^{m-3}\nabla w_{\varepsilon}\nabla u_{\varepsilon}\\
&\leq\int_{\Omega}(u_{0\varepsilon }+\varepsilon)^{m-1}+(m-1)(m-2)\int_{0}^{t}\int_{\Omega}|f'_{\varepsilon}(u_{\varepsilon})\nabla w_{\varepsilon}||(u_{\varepsilon}+\varepsilon)^{m-2}\nabla u_{\varepsilon}|\\
&\leq\int_{\Omega}(u_{0\varepsilon}+\varepsilon)^{m-1}+(m-2)\int_{0}^{t}\int_{\Omega}|f'_{\varepsilon}(u_{\varepsilon})\nabla w_{\varepsilon}||\nabla(u_{\varepsilon}+\varepsilon)^{m-1}|\\
&\leq C(u_0)+\frac{m(m-2)}{2(m-1)}\int_{0}^{t}\int_{\Omega}|\nabla(u_{\varepsilon}+\varepsilon)^{m-1}|^2+C(m)\int_{0}^{t}\int_{\Omega}|\nabla w_{\varepsilon}|^2.
\end{split}
\end{equation*}
An application of Lemma \ref{lem2.2} therefore yields $(\ref{e220})$.

Subcase (ii): $m<2$. According to H\"older's inequality, we can find positive constants $C(|\Omega|)$ such that
\begin{equation}\label{e222}
\int_{\Omega}(u_\varepsilon+\varepsilon)^{m-1}\leq C(|\Omega|)\int_{\Omega}(u_\varepsilon+\varepsilon),
\end{equation}
and we thus obtain from (\ref{e221}) that
\begin{equation*}
\begin{split}
&-\frac{m(m-2)}{m-1}\int^t_0\int_{\Omega}|\nabla (u_\varepsilon+\varepsilon)^{m-1}|^2\\
&\leq \int_{\Omega}(u_\varepsilon(\cdot,t)+\varepsilon)^{m-1}+(m-1)(m-2)\int_0^t\int_{\Omega}u_\varepsilon(u_\varepsilon+\varepsilon)^{m-3}\nabla u_\varepsilon\nabla w_\varepsilon f'_\varepsilon(u_\varepsilon)\\
&\leq C(|\Omega|)\int_{\Omega}(u_\varepsilon(\cdot,t)+\varepsilon)+(m-1)(2-m)\int_0^t\int_{\Omega}|(u_\varepsilon+\varepsilon)^{m-2}\nabla u_\varepsilon||f'_\varepsilon(u_\varepsilon)\nabla w_\varepsilon|\\
&\leq C(|\Omega|)\int_{\Omega}(u_\varepsilon(\cdot,t)+\varepsilon)+(2-m)\int_0^t\int_{\Omega}|\nabla(u_\varepsilon+\varepsilon)^{m-1}||f'_{\varepsilon}(u_\varepsilon)\nabla w_\varepsilon|\\
&\leq C(|\Omega|)\int_{\Omega}(u_\varepsilon(\cdot,t)+\varepsilon)+\frac{m(2-m)}{2(m-1)}\int^t_0\int_{\Omega}|\nabla(u_\varepsilon+\varepsilon)^{m-1}|^2+C(m)\int_0^t\int_{\Omega}|\nabla w_\varepsilon|^2.
\end{split}
\end{equation*}
An application of $(\ref{e211})$ and Lemma \ref{lem2.2} therefore yields $(\ref{e220})$.
\end{proof}
\begin{lem}\label{lem2.4}
Let $m>1$, then for all $p\in[1,\frac{2n}{n-2}]$ (or $1\leq p<\infty$ if $n=2$) and any $T>0$, there exists $C(T)>0$ such that for any $\varepsilon\in (0,1)$,
\begin{equation}\label{e223}
\int_{0}^{T}\|(u_\varepsilon+\varepsilon)^{m-1}\|_{L^p(\Omega)}dt\leq C(T).
\end{equation}
\end{lem}

\begin{proof}
We split into two cases.

Case 1: $m>2$. we have the Gagliardo-Niernberg inequality
\begin{equation*}
\begin{split}
\|(u_\varepsilon+\varepsilon)^{m-1}\|_{L^p(\Omega)}&\leq c_1\|\nabla (u_\varepsilon+\varepsilon)^{m-1}\|_{L^2(\Omega)}^\theta\|(u_\varepsilon+\varepsilon)^{m-1}\|_{L^{\frac{1}{m-1}}(\Omega)}^{1-\theta}+\|(u_\varepsilon+\varepsilon)^{m-1}\|_{L^{\frac{1}{m-1}}(\Omega)}\\
&\leq c_1(\|\nabla (u_\varepsilon+\varepsilon)^{m-1}\|_{L^2(\Omega)}^\theta+1),
\end{split}
\end{equation*}
where
\begin{equation*}
\theta=\frac{m-1-\frac{1}{p}}{m-\frac{3}{2}+\frac{1}{n}}.
\end{equation*}
On integration in time we hence find that, with some $C>0$, we have
\begin{equation*}
\int_{0}^{T}\|(u_\varepsilon+\varepsilon)^{m-1}\|_{L^p(\Omega)}^{\frac{2}{\theta}}dt\leq C\int_{0}^{T}(\|\nabla (u_\varepsilon+\varepsilon)^{m-1}\|_{L^2(\Omega)}^2+1)\leq C(1+T)
\end{equation*}
and then conclude the proof with the application of H\"older's inequality.

Case 2: $m\leq 2$. According to the Gagliardo-Niernberg inequality with $(\ref{e222})$, we can find positive constants $c_2>0$ such that
\begin{equation*}
\begin{split}
\|(u_\varepsilon+\varepsilon)^{m-1}\|_{L^p(\Omega)}&\leq c_2\|\nabla (u_\varepsilon+\varepsilon)^{m-1}\|_{L^2(\Omega)}^\theta\|(u_\varepsilon+\varepsilon)^{m-1}\|_{L^1(\Omega)}^{1-\theta}+\|(u_\varepsilon+\varepsilon)^{m-1}\|_{L^1(\Omega)}\\
&\leq c_2(\|\nabla (u_\varepsilon+\varepsilon)^{m-1}\|_{L^2(\Omega)}^\theta+1),
\end{split}
\end{equation*}
where
\begin{equation*}
\theta=\frac{1-\frac{1}{p}}{\frac{1}{2}+\frac{1}{n}}.
\end{equation*}
On integration in time we hence find that, with some $C>0$, we have
\begin{equation*}
\int_{0}^{T}\|(u_\varepsilon+\varepsilon)^{m-1}\|_{L^p(\Omega)}^{\frac{2}{\theta}}dt\leq C\int_{0}^{T}(\|\nabla (u_\varepsilon+\varepsilon)^{m-1}\|_{L^2(\Omega)}^2+1)\leq C(1+T),
\end{equation*}
and then conclude the proof with the application of H\"older's inequality.
\end{proof}
\begin{lem}\label{lem2.5}
Let $m>1$, for all $T>0$ there exists $C(T)>0$ such that for any $\varepsilon\in (0,1)$,
\begin{equation}\label{e224}
\int_{0}^{T}\|\partial_{t}(u_{\varepsilon}+\varepsilon)^{m-1}\|_{(W^{1,n+1}(\Omega))^{*}}dt\leq C(T)
\end{equation}
and
\begin{equation}\label{e225}
\int_{0}^{T}\|w_{\varepsilon t}(\cdot,t)\|_{(W^{2,2}(\Omega))^{*}}dt\leq C(T).
\end{equation}
Moreover, for each $q>2$ and any $T>0$ one can find $C(p,T)>0$ fulfilling
\begin{equation}\label{e226}
\int_{0}^{T}\|v_{\varepsilon t}(\cdot,t)\|^2_{(W^{1,q}(\Omega))^{*}}dt\leq C(q,T)
\end{equation}
for each $\varepsilon\in (0,1)$.
\end{lem}
\begin{proof}
On testing the first equation in $(\ref{e216})$ by $(u_{\varepsilon}+1)^{m-2}\psi$ for fixed $t>0$ and arbitrary $\psi\in C^{\infty}(\bar{\Omega})$, we obtain
\begin{equation}\label{e227}
\begin{split}
\frac{1}{m-1}\int_{\Omega}\partial_{t}(u_{\varepsilon}+\varepsilon)^{m-1}\cdot\psi&=-\frac{m(m-2)}{(m-1)^2}\int_{\Omega}|\nabla (u_{\varepsilon}+\varepsilon)^{m-1}|^2\psi\\
&-\frac{m}{m-1}\int_{\Omega}(u_{\varepsilon}+\varepsilon)^{m-1}\nabla (u_{\varepsilon}+\varepsilon)^{m-1}\cdot\nabla\psi\\
&-\frac{m-2}{m-1}\int_{\Omega}\frac{u_{\varepsilon}}{u_\varepsilon+\varepsilon}\nabla (u_{\varepsilon}+\varepsilon)^{m-1}f'_{\varepsilon}(u_{\varepsilon})\nabla w_{\varepsilon}\psi\\
&-\int_{\Omega}\frac{u_{\varepsilon}}{u_\varepsilon+\varepsilon}(u_{\varepsilon}+\varepsilon)^{m-1}f'_{\varepsilon}(u_{\varepsilon})\nabla w_{\varepsilon}\cdot\nabla\psi.
\end{split}
\end{equation}
Then by several applications of the Cauchy-Schwarz inequality we implies that
\begin{equation*}
\begin{aligned}
\left|\int_{\Omega}\partial_{t}(u_{\varepsilon}+\varepsilon)^{m-1}\cdot\psi\right|&\leq C(m)\left\{\int_{\Omega}|\nabla (u_{\varepsilon}+\varepsilon)^{m-1}|^2\right.\\
&\left.+\left(\int_{\Omega}|(u_{\varepsilon}+\varepsilon)^{m-1}|^p\right)^{\frac{1}{p}}\cdot\left(\int_{\Omega}|\nabla (u_{\varepsilon}+\varepsilon)^{m-1}|^2\right)^{\frac{1}{2}}\right.\\
&\left.+\left(\int_{\Omega}|\nabla(u_{\varepsilon}+\varepsilon)^{m-1}|^2\right)^{\frac{1}{2}}\cdot\left(\int_{\Omega}|\nabla w_{\varepsilon}|^2\right)^{\frac{1}{2}}\right.\\
&\left.+\left(\int_{\Omega}|(u_{\varepsilon}+\varepsilon)^{m-1}|^p\right)^{\frac{1}{p}}\cdot\left(\int_{\Omega}|\nabla w_{\varepsilon}|^2\right)^{\frac{1}{2}}\right\}\cdot\{\|\psi\|_{L^\infty(\Omega)}+\|\nabla \psi\|_{L^n(\Omega)}\}
\end{aligned}
\end{equation*}
for all $\psi\in C^{\infty}(\bar{\Omega})$, where $p=\frac{2n}{n-2}$ and we have used that $0\leq f'_{\varepsilon}\leq 1$ by (\ref{e28}). Then by Young's inequality  we can imply that
\begin{equation*}
\begin{split}
&\left|\int_{\Omega}\partial_{t}(u_{\varepsilon}+\varepsilon)^{m-1}\cdot\psi\right|\\
&\leq \left\{\int_{\Omega}|\nabla (u_\varepsilon+\varepsilon)^{m-1}|^2+\left(\int_{\Omega}|(u_\varepsilon+\varepsilon)^{m-1}|^p\right)^\frac{2}{p}+\int_{\Omega}|\nabla w_\varepsilon|^2\right\}\cdot\{\|\psi\|_{L^\infty(\Omega)}+\|\nabla \psi\|_{L^n(\Omega)}\}.
\end{split}
\end{equation*}
Since in view of the fact that $W^{1,n+1}(\Omega)\hookrightarrow L^{\infty}(\Omega)$ we can fix $c_1>0$ such that $\|\nabla\psi\|_{L^n(\Omega)}+\|\psi\|_{L^\infty(\Omega)}\leq c_1\|\psi\|_{W^{1,n+1}(\Omega)}$ for any such $\psi$, this entails that
\begin{equation*}
\begin{split}
&\|\partial_t\ln(u_\varepsilon(\cdot,t)+1)\|_{(W^{1,n+1}(\Omega))^{*}}\\
&\leq c_1\left\{\int_{\Omega}|\nabla (u_\varepsilon+\varepsilon)^{m-1}|^2+\left(\int_{\Omega}|(u_\varepsilon+\varepsilon)^{m-1}|^p\right)^\frac{2}{p}+\int_{\Omega}|\nabla w_\varepsilon|^2\right\}\ \ \ \text{for all}\ t>0.
\end{split}
\end{equation*}
 Therefore after an integration in time, thanks to Lemmas \ref{lem2.2}, \ref{lem2.3} and \ref{lem2.4} this implies $(\ref{e224})$.

Similarly, for $\psi\in C^{\infty}(\bar{\Omega})$ and $t>0$ we obtain from the second equation in $(\ref{e216})$ together with $(\ref{e28})$ and $(\ref{e211})$ that
\begin{equation*}
\begin{split}
&\left|\int_{\Omega}w_{\varepsilon t}(\cdot,t)\psi dx\right|\\
&=\left|-\int_{\Omega}\nabla w_\varepsilon\cdot\nabla\psi-\int_{\Omega}|\nabla w_\varepsilon|^2\psi+\int_{\Omega}f_\varepsilon(u_\varepsilon)\psi\right|\\
&\leq\bigg\{\big\{\int_{\Omega}|\nabla w_\varepsilon|^2\big\}^{\frac{1}{2}}+\int_{\Omega}|\nabla w_\varepsilon|^2+\int_{\Omega}f_\varepsilon(u_\varepsilon)\bigg\}\cdot \{\|\nabla \psi\|_{L^2(\Omega)}+\|\psi\|_{L^{\infty}(\Omega)}\}\\
&\leq\Big\{2\int_{\Omega}|\nabla w_\varepsilon|^2+\frac{1}{4}+\int_{\Omega}u_0\Big\}\cdot c_1\|\psi\|_{W^{2,2}(\Omega)}.
\end{split}
\end{equation*}
Therefore,
\begin{equation*}
\|w_{\varepsilon t}(\cdot,t)\|_{(W^{2,2}(\Omega))^*}\leq c_1\cdot\Big\{2\int_{\Omega}|\nabla w_\varepsilon|^2+\frac{1}{4}+\int_{\Omega}u_0\Big\}\ \ \text{for all}\ t>0,
\end{equation*}
from which $(\ref{e225})$ readily follows.
Finally, for fixed $q>2$ we have $W^{1,q}(\Omega)\hookrightarrow L^{\infty}(\Omega)$, which allows us to pick $c_2>0$ such that for all $\psi\in C^{\infty}(\bar{\Omega})$ we have $\|\nabla\psi\|_{L^2(\Omega)}+\|\psi\|_{L^{\infty}(\Omega)}\leq c_2\|\psi\|_{W^{1,q}(\Omega)}$ and hence
\begin{equation*}
\begin{split}
&\big|\int_{\Omega}v_{\varepsilon t}(\cdot,t)\psi dx\big|\\
&=\big|-\int_{\Omega}\nabla v_\varepsilon\cdot\nabla\psi-\int_{\Omega}f_\varepsilon(u_\varepsilon)v_\varepsilon\psi\big|\\
&\leq\bigg\{\big\{\int_{\Omega}|\nabla v_\varepsilon|^2\big\}^{\frac{1}{2}}+\int_{\Omega}f_\varepsilon(u_\varepsilon)v_\varepsilon\bigg\}\cdot \{\|\nabla \psi\|_{L^2(\Omega)}+\|\psi\|_{L^{\infty}(\Omega)}\}\\
&\leq\bigg\{\big\{\int_{\Omega}|\nabla v_\varepsilon|^2\big\}^{\frac{1}{2}}+\lambda\|v_0\|_{L^{\infty}(\Omega)}\bigg\}\cdot c_2\|\psi\|_{W^{1,q}(\Omega)}
\end{split}
\end{equation*}
by $(\ref{e27})$,$(\ref{e211})$ and $(\ref{e212})$. As a consequence,
\begin{equation*}
\|v_{\varepsilon t}(\cdot,t)\|_{(W^{1,q}(\Omega))^*}^2\leq 2c_2^2\Big\{2\int_{\Omega}|\nabla v_\varepsilon|^2+\lambda^2\|v_0\|_{L^{\infty}(\Omega)}^2\Big\}\ \ \text{for all}\ t>0,
\end{equation*}
so that also $(\ref{e226})$ results from Lemma \ref{lem2.2}, because $|\nabla v_\varepsilon|\leq|\nabla w_\varepsilon|\cdot\|v_0\|_{L^\infty(\Omega)}$ by $(\ref{e212})$.
\end{proof}

\begin{lem}\label{lem2.6}
There exist non-negative functions $u$ and $w$ defined on $\Omega\times (0,\infty)$ as well as a sequence $(\varepsilon_{k})_{k\in\mathbb{N}}\subset(0,1)$ such that $\varepsilon_{k}\searrow 0$ as $k\to\infty$, and such that as $\varepsilon=\varepsilon_{k}\searrow 0$,
\begin{eqnarray}
u_{\varepsilon}\to u              &\quad&a.e.\ in\ \Omega\times (0,\infty),\label{e228}\\
(u_\varepsilon+\varepsilon)^{m-1}\to u^{m-1}              &\quad&in\ L_{\rm{loc}}^2([0,\infty);L^{2}(\Omega)),\label{e229}\\
(u_\varepsilon+\varepsilon)^{m-1}\rightharpoonup u^{m-1}              &\quad&in\ L_{\rm{loc}}^2([0,\infty);W^{1,2}(\Omega)),\label{e230}\\
w_{\varepsilon}\to w              &\quad&in\ L_{\rm{loc}}^1(\bar{\Omega}\times[0,\infty))\ and\ \ a.e.\ in\ \Omega\times(0,\infty),\label{e231}\\
w_{\varepsilon}\rightharpoonup w              &\quad&in\ L_{\rm{loc}}^2([0,\infty);W^{1,2}(\Omega))\ and\label{e232}\\
w_{\varepsilon}(\cdot,t)\to w(\cdot,t)              &\quad&in\ L^1(\Omega)\ for\ a.e.\ t>0\label{e233}
\end{eqnarray}
as well as:
\begin{eqnarray}
v_{\varepsilon}\to v              &\quad&in\ L_{\rm{loc}}^1(\bar{\Omega}\times[0,\infty))\ and\ a.e.\ in\ \Omega\times(0.\infty),\label{e234}\\
v_{\varepsilon}\stackrel{\ast}{\rightharpoonup} v              &\quad&in\ L^\infty(\Omega\times(0,\infty)),\label{e235}\\
v_{\varepsilon}\rightharpoonup v              &\quad&in\ L_{\rm{loc}}^2([0,\infty);W^{1,2}(\Omega)),\label{e236}\\
v_{\varepsilon}(\cdot,t)\to v(\cdot,t)              &\quad&in\ L^1(\Omega)\ for\ a.e.\ t>0\ and\label{e237}\\
v_{\varepsilon t}\rightharpoonup v_t              &\quad&in\ L_{\rm{loc}}^2([0,\infty);(W^{1,p}(\Omega))^{\star})\ for\ all\ p>2\label{e238}
\end{eqnarray}
with $v:=\|v_0\|_{L^{\infty}(\Omega)}\cdot e^{-w}$. Moreover, the pair $(u,v)$ has the properties $(\ref{e21})$-$(\ref{e23})$ in Definition 2.1
\end{lem}
\begin{proof}
In view of Lemmas 2.2, 2.3 and 2.5, the properties (\ref{e228})-(\ref{e233}) can be achieved on choosing an appropriate sequence by means of two straightforward applications of Aubin-Lions lemma\cite{Bratanow1979Navier}. By non-negativity of $w_\varepsilon$, (\ref{e231})-(\ref{e233}) thereafter imply (\ref{e234}), (\ref{e236}) and (\ref{e237}), whereas (\ref{e212}) guarantees that on extraction of a suitable subsequence, also (\ref{e235}) is valid. Finally also (\ref{e238}) can be achieved is a direct consequence of the estimate (\ref{e226}) in Lemma \ref{lem2.4}.

Now the properties (\ref{e22}) and (\ref{e23}) immediately follow from (\ref{e228}), (\ref{e234}) and the finiteness of $w$ a.e. in $\Omega\times(0,\infty)$ as well as (\ref{e230}) and (\ref{e232}), while the second inclusion in (\ref{e21}) is obvious from (\ref{e235}) and first follow Fatou's lemma, which in conjunction with (\ref{e211}) implies that
\begin{equation*}
\int_0^T\int_{\Omega}u\leq\liminf\limits_{\varepsilon=\varepsilon_k\searrow0}u_\varepsilon\leq\lambda T
\end{equation*}
for all $T>0$.
\end{proof}
\section{Strong convergence of $u_{\varepsilon k}$}
\begin{lem}\label{lem2.7}
Let $u$ and $(\varepsilon_{k})_{k\in\mathbb{N}}\subset(0,1)$ be as obtained in Lemma $\ref{lem2.6}$. Then
\begin{equation}\label{e239}
u_{\varepsilon}\to u\ \ in\ L_{\rm{loc}}^{1}(\bar{\Omega}\times[0,\infty))\ \ as\ \varepsilon=\varepsilon_{k}\searrow 0.
\end{equation}
In particular,
\begin{equation}\label{e240}
\int_{\Omega}u(\cdot,t)=\int_{\Omega}u_{0}\ \ \ for\ a.e.\ t>0.
\end{equation}
\end{lem}
\begin{proof}
This can be derived from Lemma $\ref{lem2.6}$ by means of the Vitali convergence theorem.
  We fix $T>0$, $1\leq p\leq\frac{2n}{n-2}$ from Lemma $\ref{lem2.4}$ and $\frac{1}{q}+\frac{1}{p(m-1)}=1$. Given $\eta>0$, we can pick $\Sigma>1$ large enough and therefore $\delta>0$ suitably small such that
\begin{equation}\label{e241}
\Sigma^{\frac{1}{m-1}} T^{\frac{1}{p(m-1)}}\delta^{\frac{1}{q}}<\frac{\eta}{2}\ \ and\ \ \frac{\lambda c_1}{|\Sigma|}<\frac{\eta}{2}.
\end{equation}
Then introducing the sets:
\begin{equation*}
\begin{split}
&S_{1}(\varepsilon):=\{t\in(0,T)\big |\|(u_{\varepsilon}(x,t)+\varepsilon)^{m-1}\|_{L^p(\Omega)}\leq\Sigma\}\ and\\
&S_{2}(\varepsilon):=\{t\in(0,T)\big |\|(u_{\varepsilon}(x,t)+\varepsilon)^{m-1}\|_{L^p(\Omega)}>\Sigma\},
\end{split}
\end{equation*}
for $\varepsilon\in (0,1)$, we first use $(\ref{e223})$ to estimate
\begin{equation*}
C\geq\int_{S_2(\varepsilon)}\|(u_{\varepsilon}(x,t)+\varepsilon)^{m-1}\|_{L^p}>\int_{S_2(\varepsilon)}\Sigma=|S_2(\varepsilon)||\Sigma|
\end{equation*}
and hence
\begin{equation*}
|S_{2}(\varepsilon)|\leq\frac{C}{|\Sigma|}
\end{equation*}
for any such $\varepsilon$.
Therefore, given an arbitrary measurable $E\subset\Omega\times(0,T)$ such that $|E|<\delta$,
writing $E(t):=\{x\in\Omega|(x,t)\in E\}$ and taking $p>\frac{1}{m-1}$ we may recall $(\ref{e211})$ and twice apply the Cauchy-Schwarz inequality to see that for each $\varepsilon\in(0,1)$ we have
\begin{equation*}
\begin{split}
\int\int_{E}u_{\varepsilon}&\leq\int_{S_1(\varepsilon)}\int_{E(t)}u_{\varepsilon}(x,t)dxdt+\int_{S_2(\varepsilon)}\int_{E(t)}u_{\varepsilon}(x,t)\\
&\leq\int_{S_1(\varepsilon)}|E(t)|^{\frac{1}{q}}\Big\{\int_{\Omega}(u_{\varepsilon}(x,t)+\varepsilon)^{p(m-1)} dx\Big\}^{\frac{1}{p(m-1)}}dt+\lambda|S_2(\varepsilon)|\\
&\leq\Sigma^{\frac{1}{m-1}}\cdot|S_1(\varepsilon)|^{\frac{1}{p(m-1)}}\Big\{\int_{S_1(\varepsilon)}|E(t)|dt\Big\}^{\frac{1}{q}}+\lambda|S_2(\varepsilon)|\\
&\leq \Sigma^{\frac{1}{m-1}} T^{\frac{1}{p(m-1)}}|E|^{\frac{1}{q}}+\lambda\cdot\frac{c_1}{|\Sigma|}\\
&\leq\frac{\eta}{2}+\frac{\eta}{2}=\eta,
\end{split}
\end{equation*}
according to (\ref{e241}) and our assumption on size of E. Since we already know from Lemma \ref{lem2.6} that $u_\varepsilon\to u$ a.e in $\Omega\times(0,T)$ as $\varepsilon=\varepsilon_{k}\searrow 0$, along with the Vitail theorem this shows that in fact $u_\varepsilon\to u$ in $L^1(\Omega\times(0,T))$ and thereby establishes (\ref{e239}). The mass conservation property (\ref{e240}) is thereafter an obvious consequence of (\ref{e211}) and (\ref{e239}).
\end{proof}
\begin{lem}\label{lem2.8}
The functions $u$ and $v$ obtained in Lemma $\ref{e25}$ satisfy the identity ($\ref{e26}$) in Definition \ref{defi21} for all test functions from the class indicated there.
\end{lem}
\begin{proof}
If $(\varepsilon_{k})_{k\in\mathbb{N}}\subset(0,1)$, definition of $f_{\varepsilon}$ warrants that for each $T>0$ we have
\begin{equation*}
\begin{split}
\int_0^{\infty}\int_{\Omega}|f_{\varepsilon}(u_{\varepsilon})-u|&\leq\int_0^{\infty}\int_{\Omega}|f_{\varepsilon}(u_{\varepsilon})-f_{\varepsilon}(u)|+\int_0^{\infty}\int_{\Omega}|f_{\varepsilon}(u)-u|\\
&\leq\int_0^{\infty}\int_{\Omega}|u_{\varepsilon}-u|+\int_0^{\infty}\int_{\{u(\cdot,t)\geq\frac{1}{\varepsilon}\}}|f_{\varepsilon}(u)-u|\\
&\leq\int_0^{\infty}\int_{\Omega}|u_{\varepsilon}-u|+2\int_0^{\infty}\int_{\{u(\cdot,t)\geq\frac{1}{\varepsilon}\}}u\\
&\to 0\ \ \text{as}\ \varepsilon=\varepsilon_{k}\searrow 0,
\end{split}
\end{equation*}
due to Lemma \ref{lem2.7} and dominated convergence theorem. Then for each $\varphi$ from the class in question, in light of lemma \ref{lem2.7} combined with (\ref{e235}) we know that
\begin{equation}\label{e242}
\int_0^{\infty}\int_{\Omega}f_{\varepsilon}(u_{\varepsilon})v_{\varepsilon}\varphi\to\int_0^{\infty}\int_{\Omega}uv\varphi\ \ \ \text{as}\ \varepsilon=\varepsilon_{k}\searrow 0.
\end{equation}
Thanks to (\ref{e242}) as well as (\ref{e234}) and (\ref{e236}), it readily follows on taking $\varepsilon=\varepsilon_{k}\searrow 0$ in each of the integrals making up the corresponding weak formulation of the respective initial-boundary value sub-problem of (\ref{e210}) satisfied by $v_{\varepsilon}$ that indeed (\ref{e26}) is satisfied.
\end{proof}
\section{Proof of Theorem 1.1}
\begin{lem}\label{lem2.9}
Let $w$ and  $(\varepsilon_{k})_{k\in\mathbb{N}}$ be as given by Lemma 2.6. Then for each $T>0$ we have
\begin{equation}\label{e243}
\nabla w_{\varepsilon}\to\nabla w\ \ in\ L^{2}(\Omega\times(0,T))\ \ as\ \varepsilon=\varepsilon_{k}\searrow0.
\end{equation}
\end{lem}
\begin{proof}
We refer the reader to \cite[Corollary 2.1]{Winkler2016The} for the proof.
\end{proof}
\begin{lem}\label{lem2.10}
The couple $(u,v)$ provided by Lemma 2.6 is global generalized solution of $(\ref{e11})$ in sense of Definition 2.1.
\end{lem}
\begin{proof}
In view of Lemmas \ref{lem2.6}, \ref{lem2.7} and \ref{lem2.8}, we only need to verify (\ref{e25}). To this end, we fix an arbitrary non-negative $\varphi\in C_0^\infty(\bar{\Omega}\times[0,\infty))$ and then obtain on multiplying the first equation in (\ref{e214}) by $(u_\varepsilon+\varepsilon)^{m-2}\varphi$ when $m\neq 2$
\begin{equation}\label{e244}
\begin{split}
I_1(\varepsilon)&:=\int_{0}^{\infty}\int_{\Omega}|\nabla (u_{\varepsilon}+\varepsilon)^{m-1}|^2\varphi\\
&=\frac{m-1}{m(m-2)}\int_{0}^{\infty}\int_{\Omega}(u_{\varepsilon}+\varepsilon)^{m-1}\varphi_{t}+\frac{m-1}{m(m-2)}\int_{\Omega}u_{0}^{m-1}\varphi(\cdot,0)\\
&\quad-\frac{m-1}{m-2}\int_{0}^{\infty}\int_{\Omega}(u_{\varepsilon}+\varepsilon)^{m-1}\nabla (u_{\varepsilon}+\varepsilon)^{m-1}\cdot\nabla\varphi\\
&\quad-\frac{m-1}{m}\int_{0}^{\infty}\int_{\Omega}\frac{u_{\varepsilon}}{u_{\varepsilon}+\varepsilon}\nabla (u_{\varepsilon}+\varepsilon)^{m-1}f'_{\varepsilon}(u_{\varepsilon})\nabla w_{\varepsilon}\varphi\\
&\quad-\frac{(m-1)^2}{m(m-2)}\int_{0}^{\infty}\int_{\Omega}\frac{u_{\varepsilon}}{u_{\varepsilon}+\varepsilon}(u_{\varepsilon}+\varepsilon)^{m-1}f'_{\varepsilon}(u_{\varepsilon})\nabla w_{\varepsilon}\cdot\nabla\varphi\\
&=:I_2(\varepsilon)+I_3(\varepsilon)+I_4(\varepsilon)+I_5(\varepsilon)+I_6(\varepsilon)
\end{split}
\end{equation}
for all $\varepsilon\in(0,1)$. Here choosing $T>0$ large enough such that $\varphi\equiv0$ on $\Omega\times(T,\infty)$, we obtain from (\ref{e230}) that $(u_\varepsilon+\varepsilon)^{m-1}\rightharpoonup u^{m-1}$ in $L^2((0,T);W^{1,2}(\Omega))$ as $\varepsilon=\varepsilon_k\searrow0$, which ensures that
\begin{equation}\label{e245}
I_2(\varepsilon)\to\frac{m-1}{m(m-2)}\int_{0}^{\infty}\int_{\Omega}u^{m-1}\varphi_{t}\ \ \text{and}\ \ I_4(\varepsilon)\to-\frac{m-1}{m-2}\int_{0}^{\infty}\int_{\Omega}u^{m-1}\nabla u^{m-1}\cdot\nabla\varphi
\end{equation}
as $\varepsilon=\varepsilon_k\searrow0$, and that
\begin{equation}\label{e246}
\int_{0}^{\infty}\int_{\Omega}|\nabla u^{m-1}|^2\varphi\leq \liminf\limits_{\varepsilon=\varepsilon_k\searrow0} I_1(\varepsilon).
\end{equation}
Moreover, using that $\nabla w_\varepsilon\to\nabla w$ in $L^2(\Omega\times(0,T))$ as $\varepsilon=\varepsilon_k\searrow0$ by Lemma \ref{lem2.9}, and that this combined with the observations that $0\leq\frac{u_{\varepsilon}f'_{\varepsilon}(u_{\varepsilon})}{u_{\varepsilon}+\varepsilon}\leq1$ for all $\varepsilon\in(0,1)$ and $\frac{u_{\varepsilon}f'_{\varepsilon}(u_{\varepsilon})}{u_{\varepsilon}+\varepsilon}\to1$ a.e. in $\Omega\times(0,T)$ as $\varepsilon=\varepsilon_k\searrow0$ warrants that
\begin{equation*}
\frac{u_{\varepsilon}f'_{\varepsilon}(u_{\varepsilon})}{u_{\varepsilon}+\varepsilon}\nabla w_\varepsilon\to\nabla w=-\nabla\ln v\ \ \text{in}\ L^2(\Omega\times(0,T))
\end{equation*}
as $\varepsilon=\varepsilon_k\searrow0$, we also find that
\begin{equation}\label{e247}
I_5(\varepsilon)\to\frac{m-1}{m}\int_{0}^{\infty}\int_{\Omega}(\nabla u^{m-1}\cdot\nabla\ln v)\varphi
\end{equation}
and
\begin{equation}\label{e248}
I_6(\varepsilon)\to\frac{(m-1)^2}{m(m-2)}\int_{0}^{\infty}\int_{\Omega}u^{m-1}\nabla\ln v\cdot\nabla\varphi
\end{equation}
as $\varepsilon=\varepsilon_k\searrow0$. Collecting (\ref{e245})-(\ref{e248}),we readily see that (\ref{e25}) results from (\ref{e244}).\\
When $m=2$, (\ref{e244}) should be replaced by the identity
\begin{equation}\label{e249}
\frac{1}{2}\int_{0}^{\infty}\int_{\Omega}(u_{\varepsilon}+\varepsilon)\varphi_{t}+\frac{1}{2}\int_{\Omega}u_{0}\varphi(\cdot,0)=\int_{0}^{\infty}\int_{\Omega}(u_{\varepsilon}+\varepsilon)\nabla (u_{\varepsilon}+\varepsilon)\cdot\nabla\varphi+\frac{1}{2}\int_{0}^{\infty}\int_{\Omega}u_{\varepsilon}f'_{\varepsilon}(u_{\varepsilon})\nabla w_{\varepsilon}\cdot\nabla\varphi.
\end{equation}
According to Lemmas \ref{lem2.6} and \ref{lem2.9}, we can easily see that (\ref{eremone}) results from (\ref{e249}).
\end{proof}

\textbf{Proof of Theorem 1.1.} The claim is an obvious consequence of Lemma \ref{lem2.10}.
$\hfill\Box$

\end{document}